\documentclass{amsart}
\usepackage{graphicx}
\usepackage{amssymb,amscd,amsthm,amsxtra}
\usepackage{latexsym}
\usepackage{epsfig}

\newtheorem{thm}{Theorem}[section]

\newtheorem{lem}[thm]{Lemma}
\newtheorem{prop}[thm]{Proposition}
\theoremstyle{definition}

\theoremstyle{remark}
\newtheorem{rem}[thm]{Remark}

\begin{document}

\title[$W^{2,1}$ estimate for singular solutions]{$W^{2,1}$ estimate for singular solutions to the Monge-Amp\`{e}re equation}
\author{Connor Mooney}
\address{Department of Mathematics, Columbia University, New York, NY 10027}
\email{\tt  cmooney@math.columbia.edu}


\begin{abstract}
We prove an interior $W^{2,1}$ estimate for singular solutions to the Monge-Amp\`{e}re equation, and construct an example to show 
our results are optimal.

\vspace{3mm}

\end{abstract}
\maketitle
\section{Introduction}
Interior $W^{2,p}$ estimates for the Monge-Amp\`{e}re equation
$$\det D^2u = f \quad \text{ in } \Omega, \quad \quad u|_{\partial \Omega} = 0$$
were first obtained by Caffarelli assuming that $f$ has small oscillation depending on $p$ (see \cite{C2}).

In the case that we only have $\lambda \leq f \leq \Lambda,$
De Philippis, Figalli and Savin recently obtained interior $W^{2,1+\epsilon}$ estimates for some $\epsilon$ depending only on $n,\Lambda$ and $\Lambda$ (see \cite{DF},\cite{DFS}). 
This result is optimal in light of counterexamples due to Wang (\cite{W}) obtained by seeking solutions with the homogeneity
$$u(x,y) = \frac{1}{\lambda^{2+\alpha}}u(\lambda x,\lambda^{1+\alpha}y).$$

These can be viewed as estimates for strictly convex solutions to the Monge-Amp\`{e}re equation. Indeed, at a point $x$ where $u$ is strictly convex we 
can find a tangent plane that touches only at $x$ and lift it a little to carve out a set where $u$ has linear boundary data.

In \cite{M} we show that solutions to $\lambda \leq \det D^2u \leq \Lambda$ are strictly convex away from a singular set of Hausdorff $n-1$ dimensional measure zero, and
as a consequence we prove $W^{2,1}$ regularity for singular solutions. We also construct for any $\epsilon$ a singular
solution to $\det D^2u = 1$ in $B_1 \subset \mathbb{R}^n$ ($n \geq 3$) with a singular set of Hausdorff dimension at least $n-1-\epsilon$ which is not in $W^{2,1+\epsilon}$.
However, as $\epsilon \rightarrow 0$ these examples become arbitrarily large. 
In this paper we give a more precise, quantitative version of the work done in \cite{M} and improve the examples. Our main theorem is:

\begin{thm}\label{Main}
 Assume that 
 $$\lambda \leq \det D^2u \leq \Lambda \quad \text{ in } B_1 \subset \mathbb{R}^n, \quad \quad \|u\|_{L^{\infty}(B_1)} < K.$$
 Then for some $\epsilon(n)$ and $C(n,\lambda,\Lambda,K)$ we have $\Delta u \in L \log^{\epsilon}L$ and
 $$\int_{B_{1/2}} \Delta u \left(\log(1+\Delta u)\right)^{\epsilon} \, dx \leq C.$$
\end{thm}

We also construct an example with a singular set of Hausdorff dimension exactly $n-1$ and second derivatives not in $L \log^M L$ for $M$ large, 
showing that the main theorem is in a sense optimal and that we cannot improve our estimate on the Hausdorff dimension to $n-1-\epsilon$ for any $\epsilon$.
Since solutions in two dimensions are strictly convex, this result is interesting for $n \geq 3$.

The paper is organized as follows. In section $2$ we present some preliminaries on the geometry of sections. In section $3$ we state our key proposition 
and use it to prove Theorem \ref{Main}. In section $4$ we prove the key proposition, which is a quantitative version
of work done in \cite{M} obtained by closely examining the geometry of maximal sections. Finally, in sections $5$ and $6$ we construct
an example with a singular set of Hausdorff dimension $n-1$ and show that it gives optimality of Theorem \ref{Main}. 

\section*{Acknowledgements}

I would like to thank my thesis advisor Ovidiu Savin for his patient guidance and encouragement. 
The author was partially supported by the NSF Graduate Research Fellowship Program under grant number DGE 1144155. 


\section{Preliminaries}
Let $u: \Omega \subset \mathbb{R}^n \rightarrow \mathbb{R}$ be a convex function. Then $u$ has an associated Borel measure $Mu$, called the Monge-Amp\`{e}re measure, defined by
$$Mu(A) = |\nabla u(A)|$$
where $|\nabla u(A)|$ represents the Lebesgue measure of the image of the subgradients of $u$ in $A$ (see \cite{Gut}). We say that $u$ solves $\det D^2u = f$ in the
Alexandrov sense if 
$$Mu = f \,dx.$$

We define a section of $u$ by
$$S_h(x) = \{y \in \Omega: u(y) < u(x) + \nabla u(x) \cdot (y-x) + h\}$$
for some subgradient $\nabla u(x)$ at $x$. Finally, we define $D_{n,\lambda,\Lambda,K}$ to be the collection of convex functions satisfying
$$\lambda \leq \det D^2u \leq \Lambda \quad \text{ in } B_1 \subset \mathbb{R}^n, \quad \quad \|u\|_{L^{\infty}(B_1)} \leq K$$
in the Alexandrov sense and we say that a constant depending only on $n,\lambda,\Lambda$ and $K$ is a universal constant.
In this section we recall some geometric observations about sections of solutions in $D_{n,\lambda,\Lambda,K}$.

\begin{lem}\label{JohnsLemma}
 (John's Lemma). If $S \subset \mathbb{R}^n$ is a bounded convex set with nonempty interior, and $0$ is the center of mass of $S$, then there exists
 an ellipsoid $E$ and a dimensional constant $C(n)$ such that
 $$E \subset S \subset C(n)E.$$ 
\end{lem}

\noindent We call $E$ the John ellipsoid of $S$. There is some linear transformation $A$ such that $A(B_1) = E$, and we say that $A$ normalizes
$S$.

In the following two lemmas we present an important observation on the volume growth of sections that are not compactly contained and relate the volume of compactly contained sections
to the Monge-Amp\`{e}re mass of these sections. Short proofs can be found in \cite{M}.

\begin{lem}\label{SectionGrowth}
 Assume that $\det D^2u \geq \lambda$ in $\Omega \subset \mathbb{R}^n$. Then if $S_h(x)$ is any section of $u$, we have
 $$|S_h(x)| \leq Ch^{n/2}$$
 for some constant $C$ depending only on $\lambda$ and $n$.
\end{lem}

\noindent The proof is just a barrier by above in the John ellipsoid for $S_h(x)$.

\begin{lem}\label{Alexandrov}
 Let $v$ be any convex function on $\Omega \subset \mathbb{R}^n$ with $v|_{\partial \Omega} = 0$. Then
 $$Mv(\Omega)\,|\Omega| \geq c(n)|\min_{\Omega}v|^n.$$
\end{lem}

\noindent The proof is by comparing to the Monge-Amp\`{e}re mass of the function whose graph is the cone generated by the minimum point of $v$ and $\partial \Omega$.

Next, we recall the following geometric observation of Caffarelli for solutions to the Monge-Amp\`{e}re equation with bounded right hand side (see \cite{C1}).
It says that compactly contained sections $S_h(x)$ are balanced around $x$.

\begin{lem}\label{Balancing}
  Assume that $\lambda \leq \det D^2u \leq \Lambda$ in $\Omega \subset \mathbb{R}^n$. Then there exist
  $c,C(n,\lambda,\Lambda)$ such that for all $S_h(x) \subset\subset \Omega$, there is an ellipsoid $E$ centered at $0$ of volume $h^{n/2}$ with
  $$cE \subset S_h(x)-x \subset CE.$$
\end{lem}

Finally, we give the following engulfing and covering properties of compactly contained sections (see \cite{CG} and \cite{DFS}). In the following
$\alpha S_h(x)$ will denote the $\alpha$ dilation of $S_h(x)$ around $x$.

\begin{lem}\label{EngulfingProps}
 Assume that $\lambda \leq \det D^2u \leq \Lambda$ in $\Omega$. Then there exists $\delta > 0$ universal such that:
\begin{enumerate}
 \item If $S_h(x) \subset\subset \Omega$ then
 $$S_{\delta h}(x) \subset \frac{1}{2}S_h(x).$$
 \item Suppose that for some compact $D \subset \Omega$, we can associate to each $x \in D$ some $S_h(x) \subset\subset \Omega$. Then we can find a finite subcollection
 $\{S_{h_i}(x_i)\}_{i=1}^M$ such that $S_{\delta h_i}(x_i)$ are disjoint and
 $$D \subset \cup_{i=1}^M S_{h_i}(x_i).$$
\end{enumerate}
\end{lem}

\section{Statement of Key Proposition and Proof of Theorem \ref{Main}}

In this section we state the key proposition and use it to prove our main theorem. In \cite{M} we show that the Monge-Ampere mass of $u + \frac{1}{2}|x|^2$
in small balls around singular points is large compared to the mass of $\Delta u$. The proposition is a more precise, quantitative version of this statement
for long, thin sections. Let $\bar{h}(x) \geq 0$ be the largest $h$ such that $S_h(x) \subset\subset B_1$. We say that $S_{\bar{h}(x)}(x)$ is the maximal section at $x$.
If $\bar{h}(x) = 0$ then $x$ is a singular point.

\begin{prop}\label{MongeAmpereMass}
 If $u \in D_{n,\lambda,\Lambda,K}$, $v = u + \frac{1}{2}|x|^2$, $x \in B_{1/2}$ and $h > \bar{h}(x)$ then there exist $\eta(n)$ and $c$ universal such that for
 some $r$ with
 $$|\log r| > c|\log h|^{1/2},$$
 we have
 $$Mv(B_r(x)) > cr^{n-1}|\log r|^{\eta}.$$
\end{prop}

\begin{rem}
 Let $\Sigma$ denote the singular set of $u$, where $\bar{h} = 0$. It follows from proposition \ref{MongeAmpereMass} and a covering argument that
 $$\inf_{\delta > 0}\left\{\sum_{i=1}^{\infty} r_i^{n-1}|\log r_i|^{\eta}: \{B_{r_i}(x_i)\}_{i=1}^{\infty} \text{ cover } \Sigma, r_i < \delta\right\} = 0$$
 for some small $\eta(n)$, giving a quantitative version of the main theorem in \cite{M} for solutions to $\lambda \leq \det D^2u \leq \Lambda$.
\end{rem}

We will give a proof of Proposition \ref{MongeAmpereMass} in the next section by closely examining the geometric properties of maximal sections. 

The idea of the proof of Theorem \ref{Main} is to apply Proposition \ref{MongeAmpereMass} in the thin maximal sections, and then apply 
the $W^{2,1+\epsilon}$ estimate of \cite{DFS} in the larger sections to show the following decay of the integral of $\Delta u$ over its level sets:
\begin{equation}\label{IntegralDecay}
 \int_{\{\Delta u > t\}} \Delta u \, dx \leq \frac{C}{|\log t|^{\epsilon}},
\end{equation}
for some $\epsilon(n)$. Assuming this is true, theorem \ref{Main} follows easily by Fubini:
\begin{align*}
 \int_{B_{1/2}} \Delta u (\log(1 + \Delta u))^{\epsilon/2} \, dx &\leq C \int_{B_{1/2}} \Delta u \int_{1}^{1+\Delta u} \frac{1}{t(\log t)^{1-\epsilon/2}} \, dt \, dx \\
 &\leq C + C\int_{2}^{\infty} \frac{1}{t(\log t)^{1 - \epsilon/2}} \int_{\{\Delta u > t\}} \Delta u \, dx \, dt \\
 &\leq C + C\int_{2}^{\infty} \frac{1}{t(\log t)^{1+\epsilon/2}} \, dt \\
 &\leq C(\epsilon).
\end{align*}

To prove (\ref{IntegralDecay}), We first recall the following theorem of De Philippis, Figalli and Savin:

\begin{thm}\label{W21Epsilon}
 Assume that
 $$\lambda \leq \det D^2u \leq \Lambda \quad \text{ in } S_H(0), \quad \quad u|_{\partial S_H(0)} = 0$$
 and $B_1$ is the John ellipsoid for $S_H(0)$. Then there exist $C,\epsilon$ depending only on $\lambda,\Lambda$ and $n$ such that
 $$\int_{S_{H/2}(0) \cap \{\Delta u > t\}} \Delta u \, dx < Ct^{-\epsilon}.$$
\end{thm}

We will use the rescaled version of this theorem in the larger maximal sections.

\begin{lem}\label{RescaledW21Epsilon}
 If $u \in D_{n,\lambda,\Lambda,K}$ with $x \in B_{1/2}$ and $S_h(x) \subset\subset B_1$, then for $C$ universal and $\epsilon(n,\lambda,\Lambda)$ we have
 $$\int_{S_{h/2}(x) \cap \{\Delta u > t\}} \Delta u \, dx < Ch^{n/2-1-\epsilon}t^{-\epsilon}.$$
\end{lem}
\begin{proof}
 By subtracting a linear function and translating assume that $x=0$ and $u|_{\partial S_h(0)} = 0$. Let
 $$u(x) = (\det A)^{2/n}\tilde{u}(A^{-1}x)$$
 where $A$ normalizes $S_h(x)$ and $\tilde{u}$ has height $H$. Then
 $$D^2 u(x) = C|S_h(0)|^{2/n}(A^{-1})D^2\tilde{u}(A^{-1}x)(A^{-1})^T.$$
 Applying the estimate on $|S_h(0)|$ from Lemma \ref{Balancing} and letting $d$ denote the length of the smallest axis for the John ellipsoid of $S_h(0)$, it follows that
 $$\Delta u(x) \leq C\left(\frac{h}{d^2}\right)\Delta\tilde{u}(A^{-1}x).$$
 Using change of variables and Theorem \ref{W21Epsilon} we obtain that
 \begin{align*}
  \int_{S_{h/2}(0)\cap \{\Delta u > t\}} \Delta u \, dx &\leq C(\det A)\left(\frac{h}{d^2}\right) \int_{S_{H/2}(0) \cap \{\Delta \tilde{u} > c\frac{d^2}{h}t\}} \Delta \tilde{u}(y) \,dy \\
  &\leq C(\det A)\left(\frac{h}{d^2}\right)^{1+\epsilon}t^{-\epsilon}.
 \end{align*}
 Since $\det A = h^{n/2}$ up to a universal constants and $d > ch$ since $u$ is locally Lipschitz, the conclusion follows.
\end{proof}

Let $F_{\gamma} = \{x \in B_{1/2}: \frac{\gamma}{2} \leq \bar{h}(x) < \gamma\}.$

\begin{lem}\label{LargeSectionEstimate}
 Let $u \in D_{n,\lambda,\Lambda,K}$. Then there is some $C$ universal and $\epsilon(n,\lambda,\Lambda)$ such that
 $$\int_{F_{\gamma} \cap \{\Delta u > t\}} \Delta u \, dx < C\gamma^{-\epsilon}t^{-\epsilon}$$
\end{lem}
\begin{proof}
 By Lemma \ref{EngulfingProps} we can take a cover of $F_{\gamma}$ by sections $\{S_{\bar{h_i}(x_i)/2}(x_i)\}_{i=1}^{M_{\gamma}}$ with $x_i \in F_{\gamma}$ and $S_{\delta\bar{h_i}(x_i)}(x_i)$ disjoint
 for some universal $\delta$. Then
 $$\int_{F_{\gamma} \cap \{\Delta u > t\}} \Delta u \, dx \leq CM_{\gamma}\gamma^{n/2-1-\epsilon}t^{-\epsilon}$$
 by Lemma \ref{RescaledW21Epsilon}. We need to estimate the number of sections $M_{\gamma}$ in our Vitali cover of $F_{\gamma}$.

 Take $x \in F_{\gamma}$ and consider $S_{\bar{h}(x)}(x)$, which touches $\partial B_1$. By translation and subtracting a linear function assume that $x = 0$
 and $u|_{\partial S_{\delta^2\bar{h}(0)}(0)} = 0$. By rotating and applying Lemma \ref{Balancing} 
 assume that $S_{\delta^2\bar{h}(0)}(0)$ contains the line segment from $-ce_n$ to $ce_n$, with $c$ universal.

 Let $w_t$ be the restriction of $u$ to $\{x_n = t\}$ and let 
 $$S^{w_t} = S_{\delta^2\bar{h}(0)}(0) \cap \{x_n = t\}$$
 be the slice of $S_{\delta^2\bar{h}(0)}(0)$ at $x_n = t$. Since $|S_{\delta^2\bar{h}(0)}(0)| \leq C\gamma^{n/2}$ and this section has length $2c$ in the
 $e_n$ direction, it follows from convexity that
 $$|S^{w_t}|_{\mathcal{H}^{n-1}} \leq C\gamma^{n/2}.$$
 By convexity, $u(te_n) < -\delta^2\bar{h}(0)/2$ for $-c/2 \leq t \leq c/2$. Applying Lemma \ref{Alexandrov}, we conclude that for $t \in [-c/2,c/2]$,
 $$Mw_t(S^{w_t}) > c\gamma^{n/2-1}.$$

 Let $r$ be the distance between $\partial S_{\delta^2\bar{h}(0)}(0)$ and $\partial (2S_{\delta^2\bar{h}(0)}(0))$. Divide $2S_{\delta^2\bar{h}(0)}(0)$ into
 the slices
 $$S_k = 2S_{\delta^2\bar{h}(0)}(0) \cap \{kr < x_n < (k+1)r\}$$
 for $k = -\frac{c}{2r}$ to $\frac{c}{2r}$. Let $v = u + \frac{1}{2}|x|^2$. 
 Then $\nabla v(S_k)$ contains a ball of radius $r/2$ around each point in $\nabla v(S^{w_{(k+1/2)r}})$ (see Figure \ref{Lemma35Pic}), so
 $$Mv(S_k) \geq crMv(S^{w_{(k+1/2)r}}) \geq cr\gamma^{n/2-1}.$$

 Summing from $k = -\frac{c}{2r}$ to $\frac{c}{2r}$ we obtain that 
 $$|\nabla v(2S_{\delta^2\bar{h}(0)}(0))| \geq c\gamma^{n/2-1}.$$
 Using that $2S_{\delta^2 \bar{h}_i}(x_i) \subset S_{\delta \bar{h}_i}(x_i)$ are disjoint and summing over $i$ we obtain that
 $$M_{\gamma} \gamma^{n/2-1} < C$$
 and the conclusion follows.

\begin{figure}
 \centering
    \includegraphics[scale=0.35]{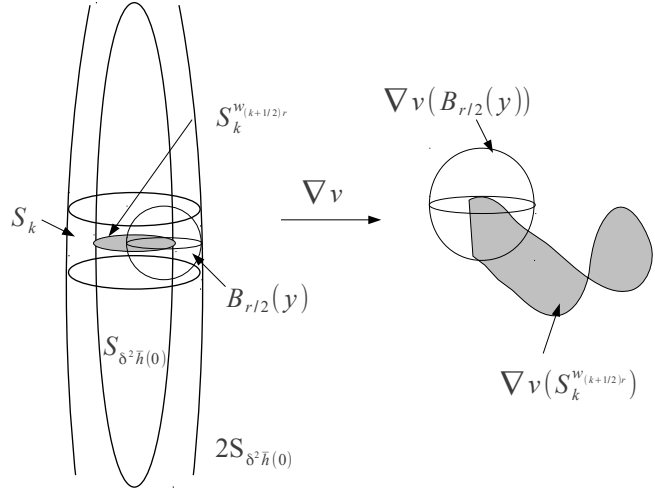}
 \caption{$\nabla v(S_k)$ contains an $r/2$-neighborhood of the surface $\nabla v(S^{w_{(k+1/2)r}})$, which projects in the $x_n$ direction to a set of 
 $\mathcal{H}^{n-1}$ measure at least $c\gamma^{n/2-1}$.}
 \label{Lemma35Pic}
\end{figure}

\end{proof}

\begin{proof}[\textbf{Proof of Theorem \ref{Main}}]
We first consider the set where $\bar{h}(x) \leq \frac{1}{t^{1/2}}$. At any point in this set, by Proposition \ref{MongeAmpereMass}, we can find
some $r > 0$ such that
$|\log r| > c|\log t|^{1/2}$
and
$$Mv(B_r(x)) > cr^{n-1}(\log t)^{\eta/2}.$$
We conclude that
$$\int_{B_r(x)} \Delta u \, dx \leq Cr^{n-1} \leq \frac{C}{(\log t)^{\eta/2}}Mv(B_r(x)).$$

Covering $\{\Delta u > t\} \cap \{\bar{h}(x) \leq \frac{1}{t^{1/2}}\}$ with these balls and taking a Vitali subcover $\{B_{r_i}(x_i)\}$, we obtain that
$$\int_{\{\Delta u > t\} \cap \{\bar{h}(x) < \frac{1}{t^{1/2}}\}} \Delta u \, dx \leq \frac{C}{(\log t)^{\eta/2}} \sum_{i} Mv(B_{r_i}(x_i)) \leq \frac{C}{(\log t)^{\eta/2}},$$
giving the desired bound over the ``near-singular'' points.

We now study the integral of $\Delta u$ over the remaining subset of $\{\Delta u > t\}$. Take $k_0$ so that 
$$2^{k_0-1} \leq t^{1/2} < 2^{k_0}.$$ 
Applying Lemma \ref{LargeSectionEstimate} we obtain that
\begin{align*}
 \int_{\{\Delta u > t\} \cap \{\bar{h}(x) > \frac{1}{t^{1/2}}\}} \Delta u \, dx &\leq \sum_{i=0}^{k_0} \int_{\{\Delta u > t\} \cap F_{2^{-i}}} \Delta u \, dx \\
 &\leq Ct^{-\epsilon}\sum_{i=1}^{k_0} 2^{\epsilon i} \\
 &\leq Ct^{-\epsilon/2},
\end{align*}
giving the desired bound.
\end{proof}

\section{Quantitative Behavior of Maximal Sections}
In this section we closely examine the geometric properties of maximal sections of solutions in $D_{n,\lambda,\Lambda,K}$ to prove Proposition
\ref{MongeAmpereMass}.

Let $u \in D_{n,\lambda,\Lambda,K}$ and fix $x \in B_{1/2}$. Then for any $h > \bar{h}(x)$, $S_{h}(x)$ is not compactly contained in $\partial B_1$. 
If $\bar{h}(x) > 0$, then by Lemma \ref{Balancing}, $S_{\bar{h}(x)}(x)$ contains an ellipsoid $E$ centered at $x$ with a long axis of universal length $2c$.

If $\bar{h}(x) = 0$ and $L$ is the tangent
to $u$ at $x$ then it is a consequence of lemma \ref{Balancing} (see \cite{C1}) that $\{u = L\}$ has no extremal points, and in particular for any 
$h > 0$ we know $S_h(x)$ contains a line segment (independent of $h$) exiting $\partial B_1$ at both ends.

By translating and subtracting a linear function assume that $x = 0$ and $\nabla u(0) = 0$. By rotating assume that $S_h(0)$ contains the line segment from $-ce_n$ to $ce_n$
for all $h > \bar{h}(0)$. For the rest of the section denote $\bar{h}(0)$ by just $\bar{h}$.

Let $w$ be the restriction of $u$ to $\{x_n = 0\}$ with sections $S_h^w$. Since $|S_{h}(0)| < Ch^{n/2}$ for all $h$ and $S_{\bar{h}}(0)$ contains a line segment of universal length
in the $e_n$ direction, we have
$$|S_{h}^w(0)|_{\mathcal{H}^{n-1}} < Ch^{n/2}$$
for $h \geq \bar{h}$.
In the following analysis we need to focus on those sections of $w$ with the same volume bound. The following property is sufficient:

\textbf{Property $F$:} We say $S_h^w(y)$ satisfies property $F$ if
$$w(y) + \nabla w(y) \cdot (-y) + h \geq \bar{h}.$$
(See Figure \ref{PropF}).

\begin{figure}
 \centering
    \includegraphics[scale=0.35]{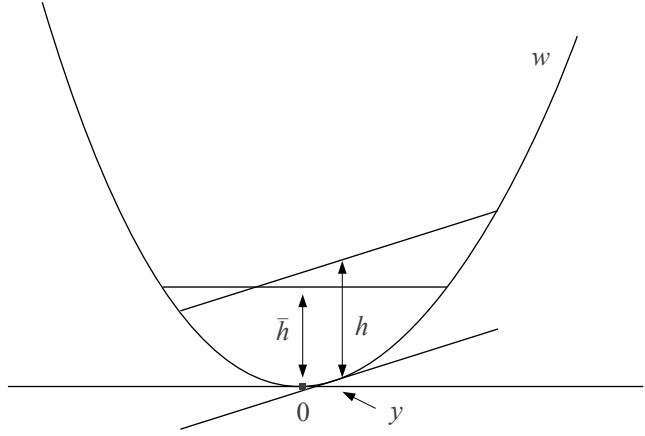}
 \caption{$S_h^w(y)$ satisfies property $F$ if the tangent plane at $y$, lifted by $h$, lies above $\bar{h}$ at $0$.}
 \label{PropF}
\end{figure}

\begin{lem}\label{TiltedVolume}
 If $S_h^w(y)$ satisfies property $F$ then
 $$|S_h^w(y)| < Ch^{n/2}.$$
\end{lem}
\begin{proof}
 The plane $u(y) + \nabla u(y) \cdot (z-y) + h$ is greater than $\bar{h}$ along $z=te_n$ for either $t > 0$ or $t < 0$. Since $u < \bar{h}$ on the segment from
 $-ce_n$ to $ce_n$, it follows that $S_h(y)$ contains the line segment from $0$ to $ce_n$ or $-ce_n$. Since $|S_h(y)| < Ch^{n/2}$ the conclusion follows.
\end{proof}

The first key lemma says that $w$ grows logarithmically faster than quadratic in at least two directions at a level
comparable to $\bar{h}$. Let 
$$d_1^y(h) \geq d_2^y(h) \geq ... \geq d_{n-1}^y(h)$$
denote the axis lengths of the John ellisoid for $S_h^w(y)$.

\begin{lem}\label{TwoBehaviors}
 For any $h > \bar{h}$ there exist $\epsilon(n)$, $C_0$ universal, $h_0 < e^{-|\log h|^{1/2}}$ and $y$ such that $S_{h_0}^w(y)$ satisfies property $F$ and
 $$d_{n-2}^y(h_0) < C_0h_0^{1/2}|\log h_0|^{-\epsilon}.$$
\end{lem}

The next lemma says that if $w$ grows logarithmically faster than quadratic in at least two directions up to height $h$
then the Monge-Amp\`{e}re mass of $u + \frac{1}{2}|x|^2$ is logarithmically larger than the mass of $\Delta u$ in a ball with radius comparable to $h^{1/2}$. 

\begin{lem}\label{Induction}
 Fix $\epsilon > 0$ and assume that for some $h > 0$, $S_h^w(y)$ satisfies property $F$. Then there exist $\eta_1,\eta_2(n,\epsilon)$ and $C$ depending on universal constants and $\epsilon$ such that if
 $$d_{n-2}^y(h) < h^{1/2}|\log h|^{-\epsilon}$$
 then for some $r < Ch^{1/2}|\log h|^{-\eta_1}$ we have
 $$M\left(u+\frac{1}{2}|x|^2\right)(B_r(0)) > C^{-1}r^{n-1}|\log r|^{\eta_2}.$$
\end{lem}

These lemmas combine to give the key proposition:
\begin{proof}[\textbf{Proof of Proposision \ref{MongeAmpereMass}:}]
 By Lemma \ref{TwoBehaviors}, there is some $S_h(y)$ satisfying property $F$ with 
 $$d_{n-2}^y(h) < C_0h^{1/2}|\log h|^{-\epsilon},$$
 with $\epsilon(n)$, $C_0$ universal and $h < e^{-|\log(\delta + \bar{h}(x))|^{1/2}}$ for any $\delta$. The conclusion follows from Lemma \ref{Induction}.
\end{proof}

We now turn to the proofs of Lemmas \ref{TwoBehaviors} and \ref{Induction}.
\begin{proof}[\textbf{Proof of Lemma \ref{TwoBehaviors}}]
 Assume by way of contradiction that for all $h < h_0$ and $S_h^w(y)$ satisfying property $F$ we have
 $$d_{n-2}^y(h) > C_0h^{1/2}|\log h|^{-\epsilon},$$
 for $h_0$ depending on $\bar{h}$ and $C_0,\epsilon$ we will choose later. We divide the proof into two steps.
 
 \textbf{Step 1:} Define the breadth $b(h)$ as the minimum distance between two parallel tangent hyperplanes to $\partial S_h^w(0)$.
 We show that for $\bar{h}|\log \bar{h}| < h < h_0$ we have
 $$b(h/2) > \left(\frac{1}{2} + \frac{C_1}{|\log h|}\right) b(h)$$
 for some $C_1$ large depending on $C_0$.
 Let $x_0$ be the center of mass of $S_h^w(0)$ and rotate so that the John ellipsoid for $S_h^w(0)$ is $A(B_1) + x_0$, where 
 $$A = \text{diag}(d_1^0(h),...,d_{n-1}^0(h)).$$
 Let $P_1,P_2$ be the tangent hyperplanes to $\partial S_{h/2}^w(0)$ a distance $b(h/2)$ apart. Let $x_1,x_2$ be points where $P_1$ and $P_2$ become tangent
 to $\partial S_h^w(0)$ when we slide them out. Assume that the distance between $0$ and the plane tangent at $x_1$ is larger than that between
 $0$ and the plane tangent at $x_2$. (See Figure \ref{Lemma42Pic1}).

\begin{figure}
 \centering
    \includegraphics[scale=0.35]{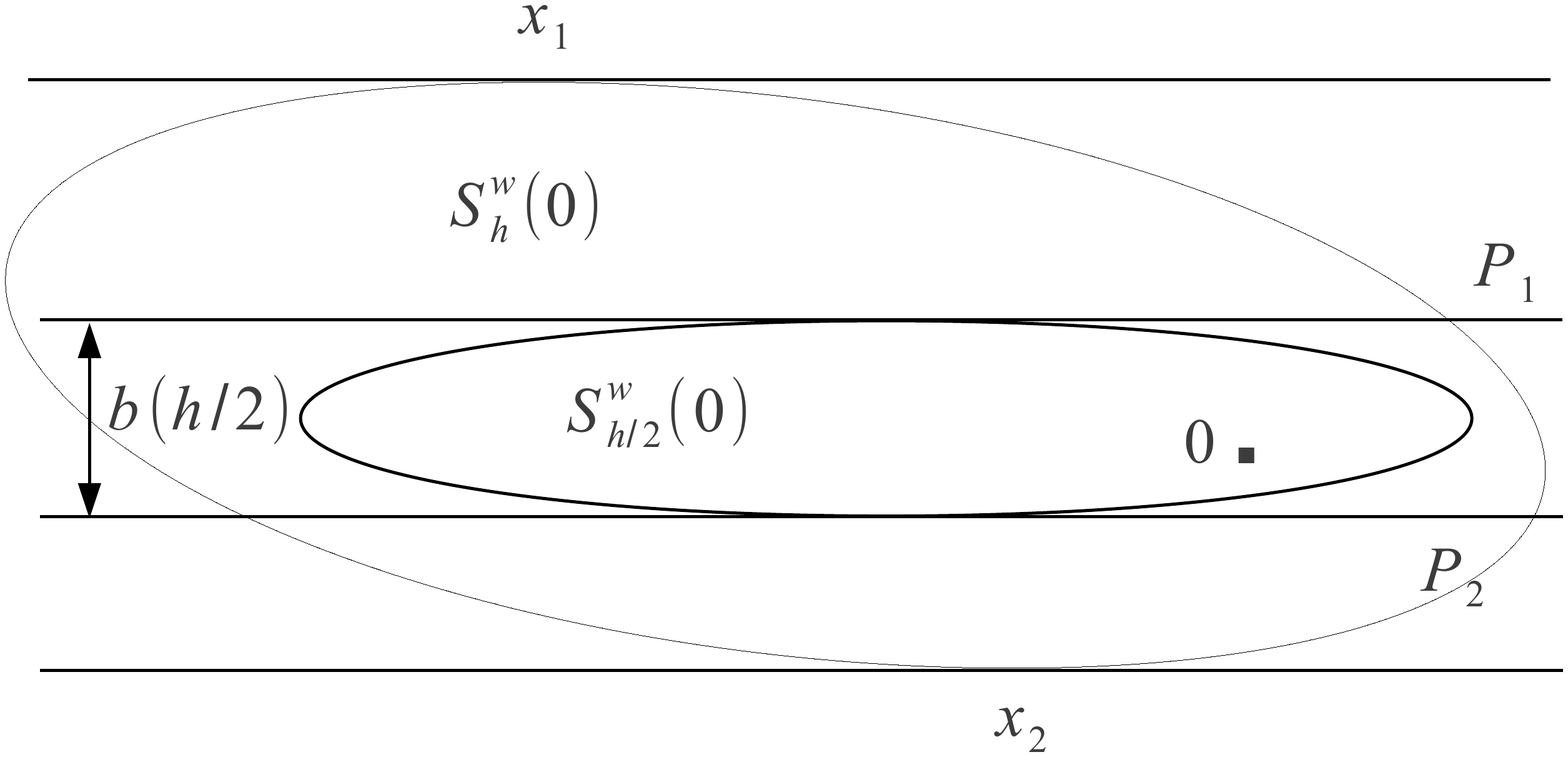}
 \caption{}
 \label{Lemma42Pic1}
\end{figure}

 Let $\tilde{x}_1$ be the image of $x_1$ under $A^{-1}$ and let
 $$\tilde{w}(x) = (\det A)^{-2/n}w(Ax).$$
 Observe that $\tilde{w}$ is the restriction of $\tilde{u}(x) = (\det A)^{-2/n}u(Ax',x_n)$ which solves $\lambda \leq \det D^2u \leq \Lambda$, so that
 sections $S_h^{\tilde{w}}$ of $\tilde{w}$ satisfying property $F$ with $\bar{h}$ replaced by 
 $(\det A)^{-2/n}\bar{h}$ have volume bounded above by $Ch^{n/2}$. Furthermore, since the distance between $0$ and the plane tangent at $x_1$ was larger and
 the images of the tangent planes under $A^{-1}$ are separated by distance at least $2$, we have $|\tilde{x}_1| \geq 1$.
 
 By convexity we can find $\tilde{y}$ on the line segment connecting $0$ to $\tilde{x}_1$ such that
 $$\nabla \tilde{w}(\tilde{y}) \cdot \frac{\tilde{x}_1}{|\tilde{x}_1|} = \frac{H}{|\tilde{x}_1|},$$
 where $H = \det A^{-2/n}h$ is the height of $\tilde{w}$. Let $\tilde{h}$ be the smallest $t$ such that $0 \in S_{t}^{\tilde{w}}(\tilde{y})$. We aim to bound
 $\tilde{h}$ below, which heuristically rules out cone-like behavior in the $\tilde{x}_1$ direction.
 Let
 $$h^* = \tilde{h} + (\det A)^{-2/n}\bar{h}.$$
 We have chosen $h^*$ so that $S_{h^*}^{\tilde{w}}(\tilde{y})$ and $S_{\delta}^w(y) = A(S_{h^*}^{\tilde{w}}(\tilde{y}))$ satisfy property $F$, where
 $\delta = (\det A)^{2/n}h^*$. (See Figure \ref{Lemma42Pic2}). It follows that 
 $$|S_{h^*}^{\tilde{w}}(\tilde{y})| < C(h^*)^{n/2}.$$

\begin{figure}
 \centering
    \includegraphics[scale=0.35]{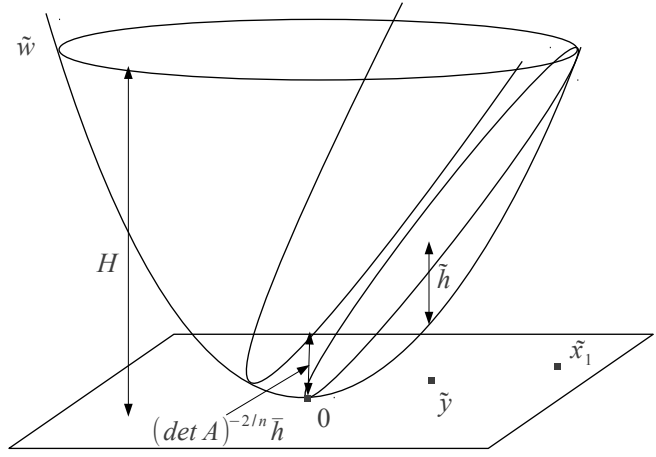}
 \caption{Lifting the tangent plane at $\tilde{y}$ by $h^* = \tilde{h} + \det(A)^{-2/n}\bar{h}$ we obtain a section of $\tilde{w}$ satisfying property $F$.}
 \label{Lemma42Pic2}
\end{figure}

 We now bound the volume of $S_{h^*}^{\tilde{w}}(\tilde{y})$ by below. Since $0,\tilde{x}_1$ are in this section, it has diameter at least $1$.
 Since $\tilde{w}$ has height $H$ it has interior Lipschitz constant $\frac{C}{H}$, so the smallest axis of the John ellipsoid for $S_{h^*}^{\tilde{w}}(\tilde{y})$ has length at least $c\frac{h^*}{H}$.
 We turn to the remaining axes.

 Let $E_y$ be the John ellipsoid for $S_{\delta}^w(y)$. By contradiction hypothesis for any $n-2$ dimensional plane $P$ passing through the center of $E_y$, we can
 find a $n-3$ dimensional plane $P'$ contained in $P$ such that $P' \cap E_y$ is an $n-3$ dimensional ellipsoid with axes $d_{1,P'}^y \geq ... \geq d_{n-3,P'}^y$ satisfying
 $$d_{n-3,P'}^y > C_0\delta^{1/2}|\log \delta|^{-\epsilon}.$$
 Take $P$ such that $A^{-1}(P)$ is perpendicular to the segment connecting $0$ and $\tilde{x}_1$.
 By using the hypothesis and that $w$ is locally Lipschitz we have 
 $$d_{n-2}^0(h)d_{n-1}^0(h) > cC_0h^{3/2}|\log h|^{-\epsilon}.$$
 Since
 $$d_1^0(h)...d_{n-1}^0(h) < Ch^{\frac{n}{2}},$$
 this gives
 $$d_{1}^0(h)...d_{n-3}^0(h) < \frac{C}{C_0}h^{\frac{n-3}{2}}|\log h|^{\epsilon}.$$
 It follows that $A^{-1}$ changes the $n-3$ dimensional volume of $P' \cap E_y$ by a factor of at least
 $$\frac{c(n)}{d_1^0(h)...d_{n-3}^0(h)} \geq cC_0h^{-\frac{n-3}{2}}|\log h|^{-\epsilon}.$$
 Since
 $$\det A > ch^{n/2}|\log h|^{-C(n)\epsilon}$$
 (by the contradiction hypothesis) and $\delta = (\det A)^{2/n}h^*$ we conclude that
 \begin{align*}
  |S_{h^*}^{\tilde{w}}(\tilde{y}) \cap A^{-1}(P')|_{\mathcal{H}^{n-3}} &> C_1\frac{(\delta^{1/2}|\log \delta|^{-\epsilon})^{n-3}}{d_{1}^0(h)...d_{n-3}^0(h)} \\
  &\geq C_1(h^*)^{\frac{n-3}{2}}(\det A)^{\frac{n-3}{n}}h^{-\frac{n-3}{2}}(C|\log h| + |\log h^*|)^{-C(n)\epsilon}
 \end{align*}
 for some large $C_1$ depending on $C_0$.
 We also have
 $$H = h(\det A)^{-2/n} \leq |\log h|^{C(n)\epsilon}.$$ 
 Using that the remaining axes have lengths  at least $1$ and $c\frac{h^*}{H}$ we obtain
 $$|S_{h^*}^{\tilde{w}}(\tilde{y})| > C_1(h^*)^{\frac{n-1}{2}}|\log h|^{-C(n)\epsilon}(C|\log h| + |\log h^*|)^{-C(n)\epsilon}.$$
 Using that $|S_{h^*}^{\tilde{w}}(\tilde{y})| < C(h^*)^{n/2}$ we get a lower bound on $h^*$:
 $$h^* > C_1|\log h|^{-C(n)\epsilon}.$$
 (See Figure \ref{Lemma42Pic3} for the simple case $n=3$.)

\begin{figure}
 \centering
    \includegraphics[scale=0.35]{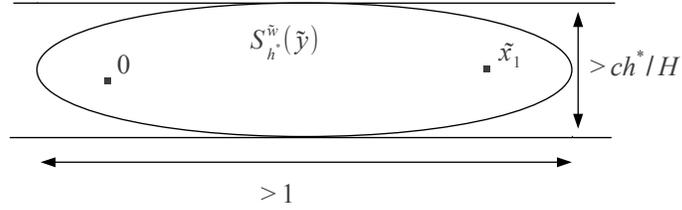}
 \caption{For the case $n=3$, the above figure implies that $|S_{h^*}^{\tilde{w}}(\tilde{y})| > ch^*/H$. This, combined with the volume estimate
 $|S_{h^*}^{\tilde{w}}(\tilde{y})| < C(h^*)^{3/2}$ and the upper bound on $H$ from the contradiction hypothesis give a lower bound
 of $c|\log h|^{-C\epsilon}$ for $h^*$.}
 \label{Lemma42Pic3}
\end{figure}

 Recalling the definition of $h^*$ and using again the lower bound on $\det A$ it follows that
 $$\tilde{h} + C\frac{\bar{h}}{h}|\log h|^{C(n)\epsilon} > C_1|\log h|^{-C(n)\epsilon}.$$
 Taking $\epsilon$ to be small enough that $C(n)\epsilon = 1/2$ and using that $\bar{h}|\log \bar{h}| < h$ we get 
 $$\tilde{h} > C_1|\log h|^{-1/2}.$$
 Finally, let $\left(\frac{1}{2}+\gamma\right)\tilde{x_1}$ be the point where $\tilde{w} = \frac{H}{2}$. 
 It is clear from convexity (see Figure \ref{Lemma42Pic4}) that 
 $$2\gamma H \geq \tilde{h}.$$

\begin{figure}
 \centering
    \includegraphics[scale=0.35]{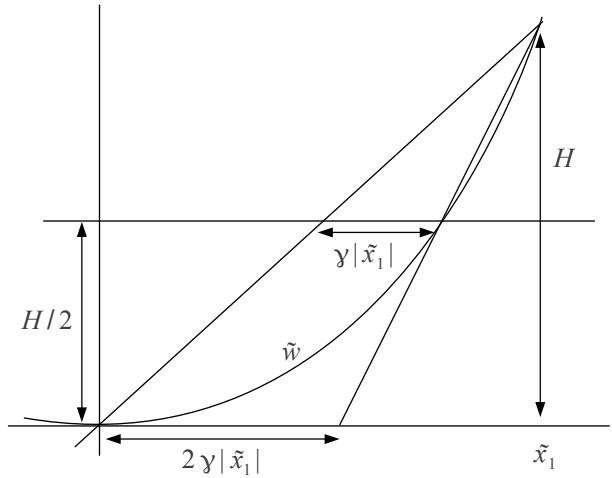}
 \caption{By convexity $2\gamma$ is at least $\tilde{h}/H$, giving a quantitative modulus of continuity for $\nabla w$ near $0$
 which we exploit in Step $2$ to obtain a contradiction.}
 \label{Lemma42Pic4}
\end{figure}

 Recalling that $H < c|\log h|^{C(n)\epsilon} < c|\log h|^{1/2}$, we obtain
 $$\gamma \geq C_1|\log h|^{-1}.$$

 Let $l_1,l_2$ be the distances from $0$ to the translations of $P_1$ and $P_2$ which are tangent to $\partial S_h^w(0)$ so that $b(h) \leq l_1 + l_2$. 
 The previous analysis implies that $P_1$ and $P_2$ have distance at least
 $\left(\frac{1}{2} + \gamma\right)l_1$ and $\frac{1}{2}l_2$ from $0$. Since $l_1 \geq l_2$ it follows that
 $$b(h/2) \geq  \left(\frac{1}{2} + \gamma\right)l_1 + \frac{1}{2}l_2 \geq \left(\frac{1 + \gamma}{2}\right)(l_1 + l_2).$$
 Since $\gamma \geq \frac{C_1}{|\log h|}$, step $1$ is finished.

 \textbf{Step 2:} We iterate step 1 to prove the lemma. First assume that $\bar{h} > 0$ and that
 $\bar{h}|\log\bar{h}| = 2^{-k}$ and $h_0 = 2^{-k_0}$.
 Note that $d_{n-1}^0(h) > c(n)b(h)$ and that $d_{n-1}^0(h_0) > c2^{-k_0}$ since $u$ is locally Lipschitz. Iterating step $1$ for
 $C_1$ large we obtain
\begin{align*}
 d_{n-1}^0(2^{-k}) &\geq c(1/2 + C_1/k)(1/2 + C_1/(k-1))...(1/2+C_1/k_0)2^{-k_0} \\
 &\geq c2^{-k}\exp(C_1\sum_{i=k_0}^k \frac{1}{i}) \\
 &\geq 2^{-k}\frac{k}{k_0},
\end{align*}
 showing that
 $$d_{n-1}^0(\bar{h}|\log\bar{h}|) \geq c\bar{h}|\log\bar{h}| \left(|\log\bar{h}||\log h_0|^{-1}\right).$$
 Finally, take $|\log h_0| = |\log \bar{h}|^{1/2}$. We conclude using convexity that
 $$d_{n-1}^0(\bar{h}) > |\log\bar{h}|^{-1}d(\bar{h}|\log\bar{h}|) > c\bar{h}|\log\bar{h}|^{1/2}.$$
 Since
 $$d_{1}^0(\bar{h})...d_{n-1}^0(\bar{h}) < C\bar{h}^{n/2}$$
 we thus have
 $$d_{n-2}^0(\bar{h}) < C\bar{h}^{1/2}|\log\bar{h}|^{-\epsilon(n)},$$
 giving the desired contradiction.

 In the case that $\bar{h} = 0$, we may run the above iteration for any $h > 0$ starting at height $h_0 = e^{-|\log h|^{1/2}}$ to obtain the
 contradiction.
\end{proof}

\begin{proof}[\textbf{Proof of Lemma \ref{Induction}}]
 First assume that $d_{1}^y(h) < h^{1/2}|\log h|^{-\alpha_1}$ for some $\alpha_1$. Since $|S_{h}^w(y)| < Ch^{n/2}$, Lemma \ref{Alexandrov} gives
 $$Mw(S_h^w(y)) > ch^{\frac{n-2}{2}}.$$
 Take $C(n)$ large enough that for $r = C(n)h^{1/2}|\log h|^{-\alpha_1}$,
 $$S_h^w(y) \subset B_{r/2}(0).$$
 Clearly, 
 $$M\left(\frac{1}{2}|x|^2 + w\right)(S_h^w(y)) > Mw(S_h^w(y)).$$
 Furthermore, $\nabla \left(u + \frac{1}{2}|x|^2\right)(B_r(0))$ contains a ball of radius
 $r/2$ around every point in $\nabla \left(u + \frac{1}{2}|x|^2\right)(S_h^w(y))$ (see Figure \ref{Lemma43Pic}). We conclude that
 \begin{align*}
  M\left(u+\frac{1}{2}|x|^2\right)(B_r(0)) &> crMw(S_h^w(y)) \\
  &\geq crh^{\frac{n-2}{2}} \\
  &\geq cr^{n-1}|\log h|^{(n-2)\alpha_1} \\
  &\geq cr^{n-1}|\log r|^{(n-2)\alpha_1}.
 \end{align*}

\begin{figure}
 \centering
    \includegraphics[scale=0.35]{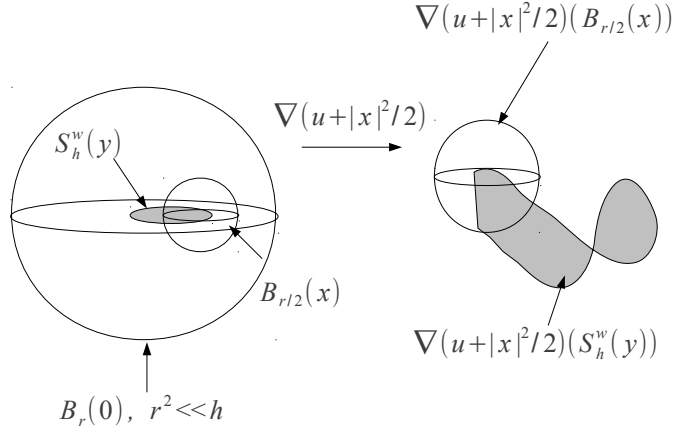}
 \caption{$\nabla (u+|x|^2/2)(B_r(0))$ contains an $r/2$-neighborhood of the surface $\nabla (u+|x|^2/2)(S_h^w(y))$, which projects in the $x_n$ direction to a set of 
 $\mathcal{H}^{n-1}$ measure at least $cr^{n-2}|\log r|^{(n-2)\alpha_1}$.}
 \label{Lemma43Pic}
\end{figure}

We proceed inductively. Assume that $d_i^y(h) > h^{1/2}|\log h|^{-\alpha_i}$ for $i = 1,...,k-1$ and that
$$d_k^y(h) < h^{1/2}|\log h|^{-\alpha_k}$$
for some $\alpha_1,...,\alpha_k$ to be chosen shortly. We aim to apply Lemma \ref{Alexandrov} to slices of the section $S_h^w(y)$ at $0$, but we need the height
of the plane $w(y) + \nabla w(y) \cdot (x-y) + h$ at $0$ to be at least $h$. We thus consider $S_{2h}^w(y)$ instead. Note that
$d_i^y(2h) > h^{1/2}|\log h|^{-\alpha_i}$ for $i \leq k-1$ and by convexity $d_k^y(2h) < 2h^{1/2}|\log h|^{-\alpha_k}$.

 Rotate so that the axes align with those for the John ellipsoid of $S_{2h}^w(y)$. Take the restriction of $w$ to the subspace spanned by $e_{k},...,e_{n-1}$, and
 call this restriction $w_k$. Let 
 $$S^{w_k} = S_{2h}^w(y) \cap \{x_1 = ... = x_{k-1} = 0\},$$
 the slice of the section $S_{2h}^w(y)$ in this subspace. Then since
 $$d_1^y(2h)...d_{n-1}^y(2h) \leq Ch^{\frac{n}{2}},$$
 by hypothesis we have
 $$|S^{w_k}|_{\mathcal{H}^{n-k}} \leq Ch^{\frac{n+1-k}{2}}|\log h|^{\alpha_1 + ... + \alpha_{k-1}}.$$
 Since $S_h^w(y)$ contains $0$ and $S^{w_k}$ is the slice of $S_{2h}^w(y)$, we know that $w_k$ has height at least $h$ in $S^{w_k}$. Using this and Lemma \ref{Alexandrov},
 $$Mw_k(S^{w_k}) \geq ch^{\frac{n-k-1}{2}}|\log h|^{-(\alpha_1 + ... + \alpha_{k-1})}.$$
 Finally, take $C(n)$ large enough that for $r = C(n)h^{1/2}|\log h|^{-\alpha_k}$ we have
 $$S^{w_k} \subset B_{r/2}(0).$$
 By strict quadratic growth, $\nabla \left(u + \frac{1}{2}|x|^2\right)(B_{r}(0))$ contains a ball of radius $r/2$ around every point in $\nabla (u + \frac{1}{2}|x|^2)(S^{w_k})$. It follows that
 \begin{align*}
  M\left(u+\frac{1}{2}|x|^2\right)(B_{r}(0)) &\geq cMw_k(S^{w_k})r^{k} \\
  &\geq ch^{\frac{n-k-1}{2}}|\log h|^{-(\alpha_1 + ... + \alpha_{k-1})}r^{k} \\
  &\geq cr^{n-1}|\log r|^{(n-k-1)\alpha_k - (\alpha_1 + ... + \alpha_{k-1})}.
 \end{align*}
Choose $\beta_i$ so that $(n-k-1)\beta_k-(\beta_1 + ... + \beta_{k-1}) = 1$ and let $\alpha_i = c\beta_i$, with $c$ chosen so that $\alpha_{n-2}=\epsilon$.
If $d_{1}^y(h) < h^{1/2}|\log h|^{-\alpha_1}$, we are done by the first step, so assume not. Then apply the inductive step for $i = 2,...,n-2$ to conclude
the proof.
\end{proof}


\section{Example}

In this section we construct a solution to $\det D^2u = 1$ in $\mathbb{R}^3$ such that $\Sigma$ has Hausdorff dimension exactly $2$. A
small modification gives the analagous example in $\mathbb{R}^n$ with a singular set of Hausdorff dimension $n-1$. This shows that the estimate on the 
Hausdorff dimension of the singular set in \cite{M} cannot be improved to $n-1-\delta$ for any $\delta$.

We proceed in several steps:
\begin{enumerate}
 \item The key step is to construct a subsolution $w$ in $\mathbb{R}^3$ satisfying $\det D^2w \geq 1$ that degenerates along $\{x_1 = x_2 = 0\}$ and grows
 logarithmically faster than quadratic in the $x_1$ direction, in particular like $x_1^2|\log x_1|^4$.
 \item Next, we construct $S \subset [-1,1]$ of Hausdorff dimension $1$ and a convex function $v$ on $[-1,1]$ such that $v$ separates
 from its tangent line faster than $r^2|\log r|^4$ at each point in $S$.
 \item Finally, we obtain our example by solving the Dirichlet problem
 $$\det D^2u = 1 \quad \text{ in } \Omega = \{|x'| < 1\} \times (-1,1), \quad \quad u|_{\partial \Omega} = C(v(x_1) + |x_2|)$$
 and comparing with $w$ at points in $S \times \{0\} \times \{\pm 1\}$.
\end{enumerate}

In the following analysis $c,C$ will denote small and large constants respectively.

\vspace{5mm}
\textbf{Construction of $w$:}
We first seek a function with just faster than quadratic growth in one direction and sections $S_h(0)$ with volume smaller than $h^{3/2}$. To that end, let
$$g(x_1,x_2) = x_1^2|\log x_1|^{\alpha} + \frac{|x_2|}{|\log x_2|^{\beta}}$$
for some $\alpha,\beta$ to be chosen shortly. It is tempting to guess $w = g(x_1,x_2)(1+x_3^2)$. However, the dominant terms in the determinant of the Hessian near the $x_2$ axis are
$$\frac{|\log x_1|^{\alpha}}{|\log x_2|^{2\beta}}\left(\frac{1}{|\log g|} - x_3^2\right),$$
where the first comes from the diagonal entries and the second from the mixed derivatives. Thus, this function is not convex. 
This motivates the following modification:
$$w(x',x_3) = g(x')\left(1+\frac{x_3^2}{|\log g(x')|}\right).$$
It is straightforward to check that the leading terms in the determinant of the Hessian (taking $x_3$ small) are
$$\frac{x_1^2|\log x_1|^{2\alpha}}{|x_2 (\log x_2)^{\beta+1} \log g|} + \frac{|\log x_1|^{\alpha}}{|(\log x_2)^{1+2\beta}\log g|},$$
since now the mixed derivative terms have the same homogeneity in $\log(g)$ as the diagonal terms. For $|x'|$ small, the first term is large in $\{|x_2| < |x_1|^3\}$, and by taking $\alpha = 2+2\beta$ the second term is bounded below by a positive constant
in $\{|x_2| \geq |x_1|^3\}$. Thus, up to rescaling and multiplying by a constant we have
$$\det D^2w \geq 1$$
in $\Omega = \{|x'| < 1\} \times (-1,1)$.
For convenience, we take $\beta = 1$ and $\alpha = 4$ for the rest of the example.

\vspace{5mm}
\textbf{Construction of $S$:}
Start with the interval $[-1/2,1/2]$. For the first step remove an open interval of length $\frac{5}{6}$ from the center.
At the $k^{th}$ step, remove intervals a fraction $\frac{5}{k+5}$ of the length of the remaining $2^k$ intervals from their centers. Denote the centers of the
removed intervals by $\{x_{i,k}\}_{i=1}^{2^k}$, and the intervals by $I_{i,k}$. Finally, let
$$S = [-1,1] - \cup_{i,k} I_{i,k}.$$
Let $l_k = |I_{i,k}|$. It is easy to check
\begin{align*}
 l_{k} &= \frac{10}{k+5} 2^{-k}\left(1-\frac{5}{k+4}\right)...\left(1-\frac{5}{6}\right) \\
 &\leq \frac{C}{k^{6}}2^{-k}.
\end{align*}

One checks similarly that the length of the remaining intervals after the $k^{th}$ step is at least
$$2^{-k}k^{-15}.$$
It follows that
\begin{equation}\label{LogHausdorffDim}
 \inf\left\{\sum_{i=1}^{\infty} r_i|\log(r_i)|^{15} : \{B_{r_i}(x_i)\} \text{ cover } S, r_i < \delta\right\} > c
\end{equation}
for all $\delta > 0$. In particular, the Hausdorff dimension of $S$ is exactly $1$.


\vspace{5mm}
\textbf{Construction of $v$:}
Let
$$ f(x) = \left\{
        \begin{array}{ll}
            |x| & \quad |x| \leq 1 \\
            2|x|-1 & \quad |x| > 1
        \end{array}
    \right.$$
We add rescalings of $f$ together to produce the desired function:
$$v(x) = \sum_{k=1}^{\infty} k^4 l_k^{2}f(l_k^{-1}(x-x_{i,k})).$$
We now check that $v$ satisfies the desired properties:
\begin{enumerate}
 \item v is convex, as the sum of convex functions. Furthermore, using that $l_k < C2^{-k}k^{-6}$ we have
 \begin{align*}
  |v(x)| &\leq C\sum_{k=1}^{\infty}\sum_{i=1}^{2^k} k^4 l_k \\
  &\leq C\sum_{k=1}^{\infty} k^{-2} \\
  &\leq C
 \end{align*}
 so $v$ is bounded.

 \item Let $x \in S$. We aim to show that $v$ separates from a tangent line more than
 $r^{2}|\log(r)|^4$ a distance $r$ from $x$. By subtracting a line assume that $v(x) = 0$ and that $0$ is a subgradient at $x$.
 Assume further that $x+r < 1/2$ and that $l_{k} < r \leq l_{k-1}$. 
 There are two cases to examine:

 \vspace{5mm}
 
 \textbf{Case 1:} There is some $y \in (x+r/2,x+r) \cap S$. Then by the construction of $S$ it is easy to see that there is some interval $I_{i,k}$ such
 that $I_{i,k} \subset (x,x+r)$. On this interval, $v$ grows by
 $$k^4l_k^2 \geq cl_k^2|\log(l_k)|^4 \geq cr^2|\log (r)|^4.$$
 \vspace{5mm}

 \textbf{Case 2:} Otherwise, there is an interval $I_{i,j}$ of length exceeding $r/2$ such that $(x+r/2,x+r) \subset I_{i,j}$. Then at the left point
 of $I_{i,j}$, the slope of $v$ jumps by at least $k^4l_k$. It follows that at $x+r$, $v$ is at least
 $$crk^4l_k \geq cr^2|\log(r)|^4.$$

 Thus, $v$ has the desired properties.
\end{enumerate}

\vspace{5mm}
\textbf{Construction of $u$:} We recall the following lemma on the solvability of the Monge-Amp\`{e}re equation (see \cite{Gut}).
\begin{lem}\label{Solvability}
 If $\Omega$ is open and convex, $\mu$ is a finite Borel measure and $\varphi$ is continuous on $\partial\Omega$
 then there exists a unique convex solution $u \in C(\bar{\Omega})$ to the Dirichlet problem
 $$\det D^2u = \mu, \quad u|_{\partial\Omega} = \varphi.$$
\end{lem}
Let $\varphi(x_1,x_2,x_3) = C(v(x_1) + |x_2|)$ for a constant $C$ we will choose shortly, 
and obtain $u$ by solving the Dirichlet problem
$$\det D^2 u = 1 \quad \text{ in } \Omega = \{|x'| < 1\} \times [-1,1], \quad \quad u|_{\partial \Omega} = \varphi.$$
Take $x \in S \times \{0\} \times \{\pm 1\}$. By translating and subtracting a linear function
assume that $x_1 = 0$ and $0$ is a subgradient for $\varphi$ at $x$. Taking $C$ large we guarantee that
$$\varphi(x_1,x_2,\pm 1) > C(x_1^2|\log(x_1)|^4 + |x_2|) > w(x_1,x_2,\pm 1)$$
for all $x_1,x_2$, and that that $\varphi > w$ on the sides of $\Omega$. Thus, $u \geq w$ in all of $\Omega$. 
Since $u = 0$ at both $(0,0,\pm 1)$ and $w(0,0,x_3) = 0$ for all $|x_3| < 1$, we have by convexity
that $u = 0$ along $(0,0,x_3)$.

This shows that for these examples
$$\Sigma \subset S \times \{0\} \times (-1,1),$$
which has Hausdorff dimension exactly $2$. 

\vspace{5mm}
\begin{rem}
 To get the analagous example in $\mathbb{R}^n$, take
 $$u(x_1,x_2,x_3) + x_4^2 + ... + x_n^2.$$
\end{rem}

\vspace{5mm}
\section{Optimality of Theorem \ref{Main}}
In \cite{M} we construct for any $\epsilon$ solutions to $\det D^2u = 1$ in $\mathbb{R}^n$ that are not 
in $W^{2,1+\epsilon}$, but as $\epsilon \rightarrow 0$ these examples blow up. 
In this section we aim to improve this by showing that the example in the previous section is not in $W^{2,1+\epsilon}$ for 
any $\epsilon$, and in fact the second derivatives are not in $L\log^{M} L$ for $M$ large.

Let $\phi(x) = (1+x)(\log(1+x))^{M}$ for some $M$ large. Then $\phi$ is convex for $x \geq 0$, so for any nonnegative integrable function $f$ and ball $B_r$ we have by 
Jensen's inequality that
$$\int_{B_r} \phi(r^nf(x)) \, dx \geq cr^n \phi\left(\int_{B_r} f(x) \, dx\right).$$
Taking $f(x) = r^{-n}\Delta u(x)$ we obtain
$$\int_{B_r} (1+\Delta u)(\log(1+\Delta u))^{M} \, dx \geq c\left(\int_{B_r} \Delta u \, dx\right)\left(\log\left(r^{-n}\int_{B_r} \Delta u \,dx\right)\right)^{M}.$$
Recall that at points $x \in S \times \{0\} \times (-1,1)^{n-2}$ the subsolutions $w$ touch $u$ by below, and that $w$ grows like $|x_2||\log x_2|^{-1}$ at $x$. It follows that
$$\sup_{\partial B_r(x)} (u-u(x)) \geq cr|\log r|^{-1}.$$
Applying convexity we conclude that
\begin{align*}
 \int_{B_r(x)}(1+\Delta u)\left(\log(1+\Delta u)\right)^{M} \, dx &\geq c\left(\int_{\partial B_r(x)} u_{\nu}\right)\left(\log\left(r^{-n}\int_{\partial B_r(x)} u_\nu\right)\right)^{M} \\
 &\geq cr^{n-1}|\log r|^{-1}\left(\log(cr^{-1}|\log(r)|^{-1})\right)^{M} \\
 &\geq cr^{n-1}|\log r|^{M - 1}.
\end{align*}
Cover $\Sigma \cap B_{1/2}$ with balls of radius less than $\delta$ and take a Vitali subcover  $\{B_{r_i}\}_{i=1}^{N}$. We then have
$$\int_{B_{1/2}} (1+\Delta u)\left(\log(1+\Delta u)\right)^{M} \, dx \geq c\sum_{i=1}^N r_i^{n-1}|\log r_i|^{M - 1},$$
and for $M$ large the right side goes to $\infty$ as $\delta \rightarrow 0$ by equation \ref{LogHausdorffDim}.

Thus, the second derivatives of $u$ are not in $L\log^{M}L$ for $M$ large, and in particular $u$ is not in $W^{2,1+\epsilon}$ for any $\epsilon$.

\vspace{5mm}



\end{document}